\documentclass[11pt,english]{amsart}
\pdfoutput=1
\usepackage[english]{babel}
\usepackage{amsthm}
\usepackage{amsmath,color}
\usepackage{amssymb}
\usepackage{graphicx}
\usepackage{dsfont}
\usepackage[bf]{caption}
\usepackage{hyperref} 
\usepackage{parskip}

\newtheorem{thm}{Theorem}[section]
\newtheorem{cor}[thm]{Corollary}
\newtheorem{lem}[thm]{Lemma}
\newtheorem{prop}[thm]{Proposition}

\newtheorem{rmk}[thm]{Remark}

\renewcommand{\qedsymbol}{$\blacksquare$}

\newcommand{\p} {\textnormal{\textsf{P}}}
\newcommand{\E} { \textnormal{\textsf{E}}}

\newcommand{\N} { \mathbb{N} }

\newcommand{\Z} { \mathbb{Z} }
\newcommand{\R} { \mathbb{R} }

\newcommand{\var}{\mathop{\mathsf{Var}}}

\newenvironment{pfof}[1]{\noindent{\underline{\textit{Proof of #1.}}}

\hspace{0.5em}}
	{\hfill\qed\vspace{1ex}}

\begin{document}

\title[Recurrence of some multidimensional RWRE]
{On the recurrence of some random walks in random environment}

\author{Nina Gantert}
\address{Nina Gantert, Technische Universit\"at M\"unchen, 
Lehrstuhl f\"ur Wahrscheinlichkeitstheorie, 
85748 Garching, 
Germany }
\email{gantert@ma.tum.de}

\author{Michael Kochler}
\address{Michael Kochler, Technische Universit\"at M\"unchen, 
Lehrstuhl f\"ur Wahrscheinlichkeitstheorie, 
85748 Garching, 
Germany }
\email{michael.kochler@tum.de}

\author{Fran\c{c}oise P\`ene}
\address{Fran\c{c}oise P\`ene, Universit\'e de Brest,
UMR CNRS 6205, Laboratoire de Math\'ematique de Bretagne Atlantique,
6 avenue Le Gorgeu, 29238 Brest cedex, France, supported by the french ANR project MEMEMO2 (ANR-10-BLAN-0125)}
\email{francoise.pene@univ-brest.fr}


\begin{abstract}
\noindent 
This work is motivated by the study of some two-dimensional random walks in random environment
(RWRE) with transition probabilities independent of one coordinate
of the walk. These are non-reversible models and can not be treated by
electrical network techniques. The proof 
of the recurrence of such RWRE
needs new estimates for quenched return probabilities
of a one-dimensional recurrent RWRE.
We obtained these estimates by constructing
suitable valleys for the potential.
They imply that $k$ independent walkers in the same one-dimensional (recurrent) environment will meet in the origin infinitely often, for any $k$.
We also consider direct products of one-dimensional recurrent RWRE with another RWRE or with a RW.
We point out the that models involving one-dimensional recurrent RWRE are
more recurrent than the corresponding models 
involving simple symmetric walk.
\end{abstract}

\maketitle

%\chapter{Strong Recurrence of Recurrent RWRE} \label{chap1}
\section{Introduction}
Since the early works of Solomon \cite{Sol} and Sinai \cite{Sin}
(see also \cite{Kesten86} and \cite{Gol}), one-dimensional random walks in random environment (RWRE)
have been studied by many authors. For an introduction to this model, we refer to \cite{Zei}. In the present work, we consider a 
one-dimensional RWRE $(X_n)_n$ with random environment 
given by a sequence
$\omega=(\omega_x)_{x\in\mathbb Z}$ of independent identically distributed (iid) random variables with values in $(0,1)$
defined on some probability space $(\Omega,\mathcal F,\p)$.
Let $z\in\mathbb Z$.
Given $\omega$, under $P_\omega^z$, $(X_n)_{n\ge 0}$ is a Markov
chain such that $P_\omega^z(X_0=z)=1$ and with the following 
transition probabilities
\begin{equation}\label{I-RWRE}
P_{\omega}^{z}(X_{n+1} = x+1|X_n=x) = \omega_x = 1 - P_{\omega}^{z}(X_{n+1} = x-1|X_n=x).
\end{equation}
For $i \in \Z$ we define
$ \rho_i = \rho_i(\omega):= \frac{1-\omega_i}{\omega_i}$ and we 
assume throughout the paper that
\begin{align}
& \E[ \log \rho_0]= 0 ,\ \ \ \var (\log \rho_0) > 0 \label{I-ass1},\\
& \p(\varepsilon \le \omega_0 \le 1 -\varepsilon) = 1 \text{ for some } \varepsilon \in \left(0, \tfrac12 \right)\label{I-ass2}.
\end{align}
The first part of \eqref{I-ass1} ensures that the RWRE is recurrent for $\p$-a.e. $\omega$,  its second part excludes the case of a deterministic environment. 
Such RWREs are often called ``Sinai's walk'' due to the results in \cite{Sin}.
Assumption
\eqref{I-ass2} (called uniform ellipticity) is a common technical condition in the context of RWRE. 
Our main results on the one-dimensional RWRE $(X_n)_n$ are the following. We write $P_\omega$ for $P^0_\omega$.

\begin{thm}\label{I-Rthm1}
For $0 \le \alpha < 1$ and for $\p$-a.e. $\omega$, we have
\begin{equation}
\label{I-thm1}
\sum_{n \in \N} P_{\omega}(X_{2n}=0) \cdot n^{-\alpha}= \infty.
\end{equation}
\end{thm}\vspace{12pt}

\begin{thm}\label{I-Rthm2}
For all $\alpha > 0$ and for $\p$-a.e. $\omega$, we have
\begin{equation}
\label{I-thm2}
\sum_{n \in \N} \Big(P_{\omega}(X_{2n}=0)\Big)^{\alpha} = \infty.
\end{equation}
In particular, $d$ independent particles performing recurrent RWRE in the same environment (and starting from the origin) are meeting in the origin infinitely often, almost surely.
\end{thm}

\begin{rmk} 
It was shown in \cite{Gal2013} that $d$ independent particles in the same environment meet infinitely often, and 
the tail of the meeting time was investigated. We show here that the $d$ particles even meet infinitely often in the origin.
\end{rmk}

%\newpage
For the next statement, we consider $d$ independent environments.

\begin{cor}\label{I-Rthm3}
For $d \in \N$, consider $d$ 
i.i.d. random environments $\omega^{(1)}, \omega^{(2)}, \ldots, \omega^{(d)}$ 
fulfilling \eqref{I-ass1} and \eqref{I-ass2}. Then, 
for $\p^{\otimes d}$-a.e. $(\omega^{(1)}, \omega^{(2)}, \ldots, \omega^{(d)})$, we have
\begin{equation}
\label{I-eqcor2}
\sum_{n \in \N} \prod_{k=1}^{d} P_{\omega^{(k)}}(X_{2n}=0) = \infty.
\end{equation}
In particular, $d$ independent particles performing recurrent RWRE in i.i.d. environments 
(and starting from the origin) 
are meeting in the origin infinitely often, almost surely.
\end{cor}

We point out that a proof of Corollary \ref{I-Rthm3} can also be found in \cite{Zei} after Lemma A.2. The proof there uses the Nash-Williams inequality in the context of electrical networks.

In \cite{CP}, Comets and Popov also consider the return probabilities of the one-dimensional recurrent RWRE on $\Z$. In contrast to our setting, they consider the corresponding jump process in continuous time $(\xi^z_t)_{t \ge 0}$ started at $z \in \Z$ and with jump rates $(\omega_x^+,\omega_x^-)_{x \in \Z}$ to the right and left neighbouring sites. One advantage of this process in continuous time is that it is not periodic as the RWRE in discrete time. 
%By considering the process $(\xi^_t)_{t \ge 0}$ in continuous time only at the random time points of the jumps, we can recover the embedded discrete-time RWRE. 
They show the following (under two conditions on the environment $(\omega_x^+,\omega_x^-)_{x \in \Z}$):

\textbf{Theorem} (cf.\ Corollary 2.1 and Theorem 2.2 in \cite{CP}) \textit{We have}
$
\frac{\log P_{\omega}(\xi_t^0 = 0)}{\log t} \xrightarrow{t \to \infty} - \widehat{a}_e$ in distribution where $\widehat{a}_e$ has the density $f$
given by $f(z)=2 - z - (z+2) \cdot e^{-2z}$ if $z\in(0,1)$
and
$f(z)=([e^2-1] \cdot z - 2 ) \cdot e^{-2z}$ if $z  \ge 1$.\medskip

Since we can embed the recurrent RWRE $(X_n)_{n \in \N_0}$ in discrete-time into the corresponding jump process in continuous time, we can expect the return probabilities to behave similarly as in the continuous setting. In particular, for $\p$-a.e.\ environment $\omega$, we expect 
$$
P_{\omega}(X_{2n} = 0) =: n^{-a(\omega,n)} 
\quad\mbox{with}\quad
 \liminf_{n \to \infty} a(\omega,n) = 0,\ \limsup_{n \to \infty} a(\omega,n) = \infty.
$$
Theorem \ref{I-Rthm1}, \ref{I-Rthm2} and Corollary \ref{I-Rthm3}
allow us to establish the recurrence of the multidimensional
RWRE $(M_n)_n$ in the cases (I)-(III) below. Except model (I)
(which is the direct products of $(X_n)_n$ with a RW), the models
considered here are 2-dimensional RWRE with transition probabilities independent of 
the vertical position of the walk. 
%This assumption
%may appear restrictive but provides actually an interesting class of RWRE
%with various behaviours. 
%The study of such models was initiated by Matheron and de Marsily in \cite{MdM} to 
%modelise transport in a stratified porus medium (see also \cite{BGKPR}) and also by Campanino and Petritis in \cite{CaPe}.

Let $\delta\in(0,1)$.
We establish recurrence of the RWRE
$(M_n)_n$ on $\mathbb Z^2$ in the three following cases:
\begin{itemize}
%\item[(I)] $d\ge 1$ and the coordinates of $(M_n)_n$ are independent 
%one-dimensional RWRE (in a same environment or in independent environments). 
\item[(I)] $d=2$ and $(M_n)_n$ is the direct product of the Sinai walk 
$(X_n)_n$ and of some recurrent random walk on $\mathbb Z$;
more precisely
$$P_\omega(M_{n+1}=(x+1,y+z)|M_n=(x,y))=\omega_x\cdot\nu(\{z\})$$
$$\mbox{and}\ \ \ P_\omega(M_{n+1}=(x-1,y+z)|M_n=(x,y))=
    (1 -\omega_x) \cdot \nu(\{z\}),$$
where $\nu$ is a distribution on $\mathbb Z$ (with zero expectation)  belonging to the
domain of attraction of an $\alpha$-stable random variable with $\alpha \in (1,2]$.
\item[(II)] $d=2$ and $(M_n)_n$ either moves horizontally with respect to the Sinai walk
(with probability $\delta$) or moves vertically with respect to some
recurrent random walk (with probability $1-\delta$):
$$P_\omega(M_{n+1}=(x+1,y)|M_n=(x,y))=\delta\omega_x=\delta - 
P_\omega(M_{n+1}=(x-1,y)|M_n=(x,y)),$$
$$P_\omega(M_{n+1}=(x,y+z)|M_n=(x,y))=(1-\delta)\cdot\nu(z), $$
where $\nu$ is a probability distribution on $\mathbb Z$
(with zero expectation) belonging to the
domain of attraction of an $\alpha$-stable distribution  with $\alpha \in (1,2]$.
%\begin{figure}[h] %!! Passt die Position hier!?
%%\begin{minipage}[t]{0.5\textwidth}
%\vspace{0pt}
%\centering
%\includegraphics[viewport=320 450 200 800, scale =0.6]{pic8a.pdf}
%  \caption{Transition probabilities for (III) in the particular 
%case when $\nu=\frac 12(\delta_{-1}+\delta_{1})$. 
%The probabilities to go up and down are always the same.
%The probabilities to go left and right depend on the horizontal position and of the environment $\omega$.}
%%\end{minipage}
%\end{figure}%\vspace*{-24pt}
\item[(III)] An odd-even oriented model: $d=2$ and $(M_n)_n$ either moves horizontally with respect to
Sinai's walk (with probability $\delta$) or moves vertically (with  probability $1-\delta$) with respect to 
$\nu$
if the first coordinate of the current position of the walk is even and to $\tilde\nu:=\nu(-\cdot)$ 
otherwise; i.e.
$$P_\omega(M_{n+1}=(x+1,y)|M_n=(x,y))=\delta\omega_x=\delta - 
P_\omega(M_{n+1}=(x-1,y)|M_n=(x,y)),$$
$$P_\omega(M_{n+1}=(x,y+z)|M_n=(x,y))=(1-\delta)\nu((-1)^xz), $$
where $\nu$ is a probability distribution on $\mathbb Z$
(admitting a first moment) such that $\nu(-\cdot)*\nu$ belongs to the
domain of attraction of a stable distribution.
\end{itemize}
\begin{figure}[h] 
\begin{minipage}[t]{0.8\textwidth}
\vspace{0pt}
\centering
\includegraphics[viewport=320 450 200 800, scale =0.6]{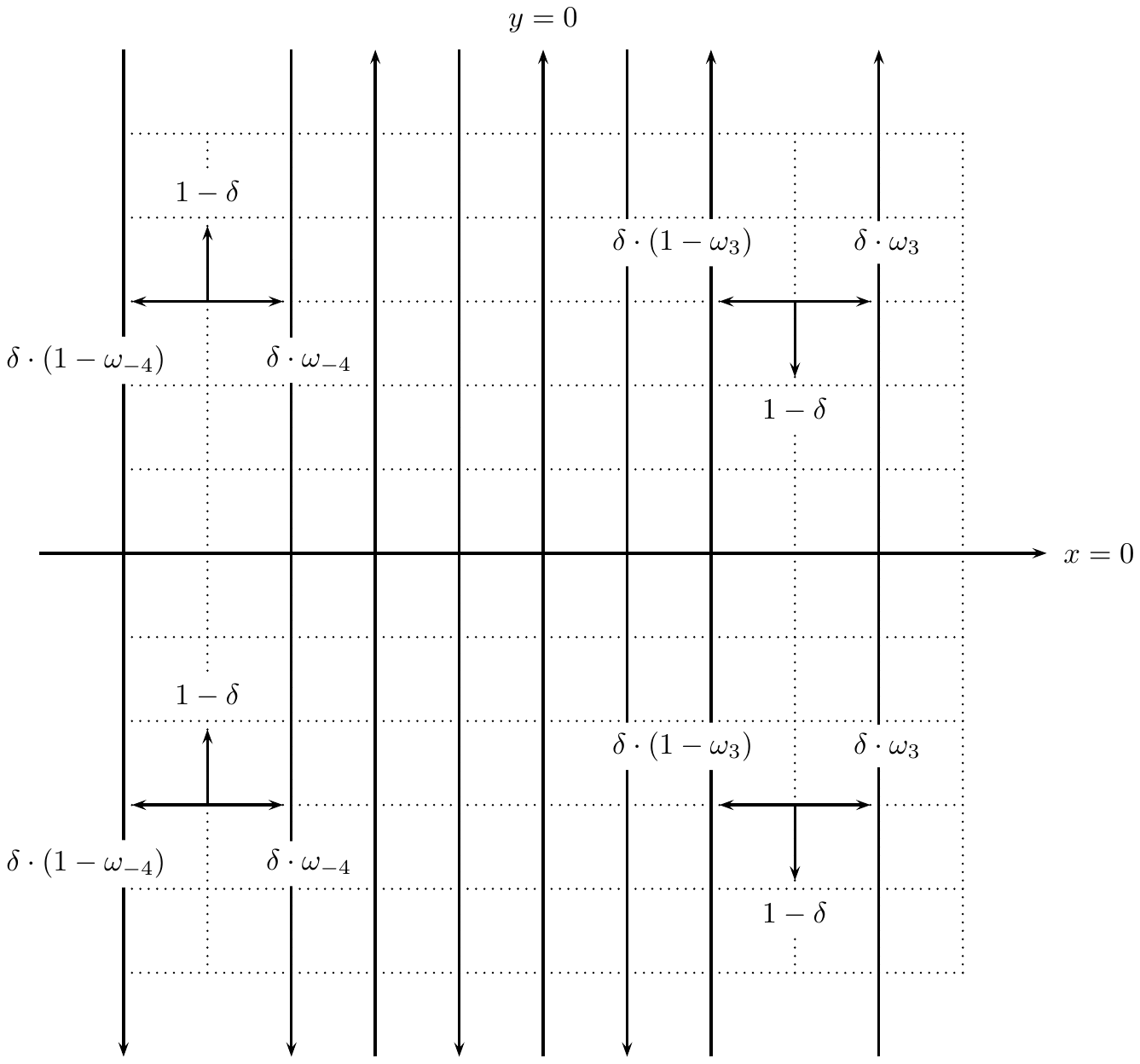}
\caption{Transition probabilities in case (III) in the particular case
where $\nu=\delta_1$. This is an example of an oriented RWRE. 
Every even vertical line is oriented upward
and every odd vertical line is oriented downward.}
\end{minipage}
\end{figure}

%\vspace*{-24pt}
If $\omega_x$ is replaced by $1/2$ (i.e. if we replace
Sinai's walk by the simple symmetric walk), the walks 
given in (I)-(III) are transient when $\nu$ is in the domain of attraction of a $\beta$-stable distribution with $\beta<2$. Hence, in this study, the Sinai's walk
gives rise to more recurrent models than the simple symmetric random walk does.

The structure of our paper is the following: In Section \ref{I-sec1.2}, we introduce the potential of the one-dimensional RWRE 
and we recall some known results. Section \ref{I-sec1.5} contains the proofs of our main results for one-dimensional RWRE. In Section \ref{I-sec1.6}, we state our recurrence results for multidimensional RWRE
involving the RWRE $(X_n)_n$ (models (I)-(III)) and we compare our results with the case 
when $(X_n)_n$ is replaced by a simple random walk. 
\section{Preliminaries} \label{I-sec1.2}
%\textbf{Bemerkung:} \textit{Die Varianz im Zusammenhang mit der Umgebung muss vermutlich benannt werden fuer das Approximation theorem!}!! Überall angepasst?
As usual, we use $P_{\omega}^{o}$ instead of $P_{\omega}^{0}$ and will even drop the superscript $o$ where no confusion is to be expected. We can now define the potential $V$ as
\begin{align} \label{I-RWRE.a}
V(x):= \begin{cases} \sum\limits_{i=1}^x \log \rho_i & \text{for }x=1,2,\ldots \\ 0 & \text{for }x=0 \\ \sum\limits_{i=x+1}^{0} -\log \rho_i & \text{for }x=-1,-2,\ldots\ . \end{cases}
\end{align}
Note that $V(x)$ is a sum of iid random variables which are centered and whose absolute value is bounded 
%are bounded by $C:=\log (1-\varepsilon) - \log \varepsilon > 0$ 
due to assumptions \eqref{I-ass1} and \eqref{I-ass2}. One of the crucial facts for the RWRE is that, for fixed $\omega$, the random walk is a reversible Markov chain and can therefore be described as an electrical network. The conductances are given by $C_{(x,x+1)}= e^{-V(x)}$
and the stationary reversible measure which is unique up to multiplication by a constant is given by
\begin{equation} \label{I-eq1.1}
\mu_{\omega}(x)= e^{-V(x)} + e^{-V(x-1)} 
\end{equation}
%= \begin{cases}  \prod\limits_{i=1}^{x-1} \frac{\omega_i}{1-\omega_i} \cdot \frac{1}{1-\omega_x} &  \text{for }x=1,2,\ldots \\ 
%\prod\limits_{i=x+1}^{0} \frac{1-\omega_i}{\omega_i} \cdot \frac{1}{\omega_x} & \text{for }x=0,-1,\ldots\ . \end{cases}
%\end{align}
The reversibility means that, for all $n \in \N_0$ and $x,y \in \Z$, we have
\begin{equation}
\label{I-eq1}
\mu_{\omega}(x) \cdot P^x_{\omega}(X_n = y) = \mu_{\omega}(y) \cdot P^y_{\omega}(X_n = x).
\end{equation}
%Now let us collect some useful properties of the RWRE. 
For the random time of the first arrival in $x$
\begin{equation}\label{I-def1}
\tau(x):= \inf \{n \ge 0:\ X_n=x\},
\end{equation}
the interpretation of the RWRE $(X_n)_n$ as an electrical network helps us to compute the following probability for $x < y < z$ (for a proof see for example formula (2.1.4) in \cite{Zei}):
\begin{equation}
\label{I-prel1}
P^y_{\omega}(\tau(z) < \tau(x)) = \frac{\sum\limits_{j=x}^{y-1} e^{V(j)}}{\sum\limits_{j=x}^{z-1} e^{V(j)}}\, .
\end{equation}
Further (cf.\ (2.4) and (2.5) in \cite{SZ} and Lemma 7 in \cite{Gol}), we have for $k \in \N$ and $y < z$
\begin{align}
\label{I-prel2}
& P^y_{\omega}(\tau(z) < k) \le k \cdot \exp \left( - \max_{y \le i < z} \big[V(z-1)-V(i) \big] \right) 
\intertext{and similarly for $x < y$}
\label{I-prel3}
& P^y_{\omega}(\tau(x) < k) \le k \cdot \exp \left( - \max_{x < i \le y} \big[V(x+1)-V(i) \big] \right).
\end{align}
To get bounds for large values of $\tau(\cdot)$, we can use that for $x < y < z$ we have (cf.\ Lemma 2.1 in \cite{SZ})
\begin{equation}
\label{I-prel4}
E_{\omega}^y[\tau(z) \cdot \mathbf{1}_{\{\tau(z) < \tau(x)\}}] \le (z-x)^2 \cdot \exp \left( \max_{x\le i \le j \le z} \big(V(j) - V(i)\big) \right).
\end{equation} 
Further, the Koml{\'o}s-Major-Tusn{\'a}dy strong approximation theorem (cf.\ Theorem 1 in \cite{KMT}, see also formula (2) in \cite{CP}) will help us to compare the shape of the potential with the path of a two-sided Brownian motion:\\

\begin{thm} \label{I-Komlos}
In a possibly enlarged probability space, there exists a version of our environment process $\omega$ and a two-sided Brownian motion $(B(t))_{t \in \R}$ with diffusion constant $\sigma:=~(\var(\log \rho_0))^{\frac12}$ (i.e.\ $Var(B(t))=\sigma^2 |t|$) such that for some $K>0$ we have
\begin{equation}
\label{I-approx}
\p \left( \limsup_{x \to \pm \infty} \frac{|V(x)-B(x)|}{\log |x|} \le K \right) =1.
\end{equation} 
\end{thm}

\section{Dimension 1~: proofs of Theorem \ref{I-Rthm1}, \ref{I-Rthm2} and Corollary \ref{I-Rthm3}} \label{I-sec1.5}
For $L \in \N$ and \mbox{$0 < \delta < 1$},
we introduce the set $\Gamma(L,\delta)$ of environments defined by
$$
\Gamma(L,\delta):=\{R^{\pm}_1(L) \le \delta L,\, R^{\pm}_2(L) \le \delta L,\, T^{\pm}(L) \le L^2\} \vphantom{\inf_{0 \le k \le T^{\pm}(L)}},$$
where
\begin{align*}
& T^{+}(L):= \inf \{z \ge 0:\ V(z) - \min_{0 \le  y \le z } V(y) \ge L\} \vphantom{\inf_{0 \le y \le T^{+}(L)}},\\
& T^{-}(L):= \sup \{z \le 0:\ V(z) - \min_{n \le y \le z} V(y) \ge L\} \vphantom{\inf_{0 \le y \le T^{+}(L)}},\\
& R^{+}_1(L):= - \min_{0 \le y \le T^{+}(L)} V(y),\ \ 
 R^{-}_1(L):= - \min_{T^{-}(L)\le y \le 0} V(y), \\
& T^{+}_b(L):= \inf \{z \ge 0:\ V(z) = - R^{+}_1(L)\}  \vphantom{\inf_{0 \le y \le T^{+}(L)}},\ \ 
 T^{-}_b(L):= \sup \{z \le 0:\ V(z) = - R^{-}_1(L)\} \vphantom{\inf_{0 \le y \le T^{+}(L)}},\\
& R^{+}_2(L):= \max_{0 \le y \le T^{+}_b(L)} V(y), \ \ 
 R^{-}_2(L):= \max_{T^{-}_b(L) \le y \le 0} V(y). \vphantom{\inf_{0 \le y \le T^{+}(L)}}
\end{align*}
We then consider the valley of the potential $V$ between $T^-(L)$
and $T^+(L)$.
Here, the $+$-sign and the $-$-sign indicate whether we deal with properties of the valley on the positive or negative half-line, respectively. Note that the definition of the set $\Gamma(L,\delta)$  is compatible with the scaling of a Brownian motion in space and time.

\begin{figure}[ht] 
\vspace{45pt} %\hspace{2cm}
\flushleft{
\ \ \ \ \ \ \ \ \ \ \ \ \ \ \ \includegraphics [viewport=105 450 380 675, scale=1.0]{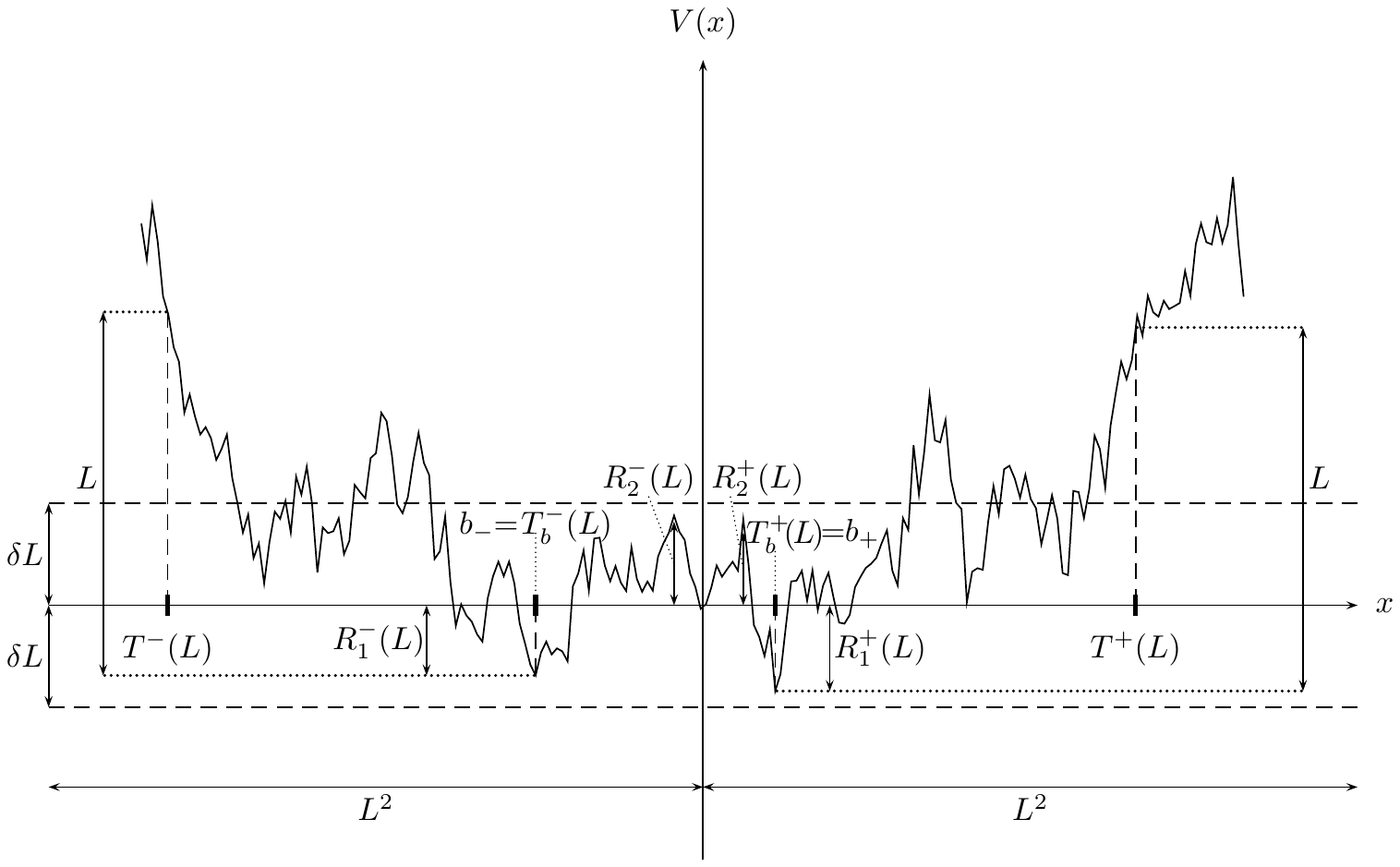} 
}
 \caption{Shape of a valley of an environment in $\Gamma(L,\delta)$} 
\label{I-figure1}
\end{figure}

\begin{rmk}
We have constructed the valleys in such a way that the return probability of the random walk to the origin is bounded from below (for even 
time points) as long as the random walk has not left the valley. If $\omega \in \Gamma(L,\delta)$, 
the random walk $(X_n)_{n \in \N_0}$ in the environment $\omega$ satisfies the following:

\begin{enumerate}
\item Since we have $V(T^{-}(L)) - V(T^{-}_b(L)) \ge L$ and $V(T^{+}(L)) - V(T^{+}_b(L)) \ge L$, the random walk $(X_n)_{n \in \N_0}$ stays within $\{T^{-}(L), T^{-}(L) + 1, \ldots, T^{+}(L)\}$ with high probability for at least $\exp((1-2\delta)L)$ steps (cf.\ \eqref{I-eq4a}).

\item Within the valley $\{T^{-}(L), T^{-}(L) + 1, \ldots, T^{+}(L)\}$, the random walk prefers to stay at positions $x$ with a small potential $V(x)$, i.e.\ at positions close to the bottom points $T^{-}_b(L)$ and $T^{+}_b(L)$.

\item The return probability for the random walk from the bottom points $T^{-}_b(L)$ and $T^{+}_b(L)$ to the origin is mainly given by the potential differences $R_2^{-}(L)+R_1^{-}(L) \le 2 \delta L$ and $R_2^{+}(L)+R_1^{+}(L) \le 2 \delta L$ respectively, i.e.\ by the height of the potential the random walk has to overcome from the bottom points back to the origin (cf.\ \eqref{I-eq2}). 
\end{enumerate}

\end{rmk}

\begin{prop} \label{I-prop1}
For every $\delta\in(0,\tfrac15)$, 
there exists $C=C(\delta)$ such that, 
for every $L$, every for $\omega \in \Gamma(L,\delta)$
and every $n$ satisfying $e^{3 \delta L} \le n \le e^{(1-2\delta)L}$, we have
\begin{equation}
\label{I-lem1}
P_{\omega}^o (X_{2n}=0) \ge C \cdot \exp(-3\delta L).
\end{equation}
\end{prop}
\begin{pfof}{Proposition \ref{I-prop1}}
The return probability to the origin for the time points of interest is mainly influenced by the shape of the ``valley'' of the environment $\omega$ between $T^-(L)$ and $T^+(L)$. For the positions of the two deepest bottom points of this valley on the positive and negative side, we write $b_\pm:=T^\pm_b(L)$ 
and we assume for the following proof that we have (cf.\ \eqref{I-def1} for the definition of $\tau(\cdot)$)
\begin{equation}
\label{I-ass4}
P^o_{\omega} \big( \tau(b_+) < \tau(b_-) \big) \ge \frac12.
\end{equation}
(Due to the symmetry of the RWRE, the proof also works in the opposite case if we switch the roles of $b_+$ and $b_-$). We have
\begin{align}
& P^o_{\omega} (X_{2n}=0) \ge P^o_{\omega} \left(X_{2n}=0,\ \tau(b_+) \le \tfrac{2n}{3},\ \tau(b_+) < \tau(b_-) \right) \nonumber \vphantom{\frac{\mu_{\omega}(0)}{\mu_{\omega}(b_+)}}\\
\ge\ & P^o_{\omega} \left(\tau(b_+) \le \tfrac{2n}{3},\ \tau(b_+) < \tau(b_-) \right) \cdot \widehat{\inf}_{\ell \in \big\{\left\lceil \tfrac{4n}{3}\right\rceil,\ldots,2n\big\}} P^{b_+}_{\omega} (X_{\ell} = 0) \nonumber \vphantom{\frac{\mu_{\omega}(0)}{\mu_{\omega}(b_+)}}\\
=\ & P^o_{\omega} \left(\tau(b_+) \le \tfrac{2n}{3},\ \tau(b_+) < \tau(b_-) \right) \cdot \frac{\mu_{\omega}(0)}{\mu_{\omega}(b_+)} \cdot \widehat{\inf}_{\ell \in \big\{\left\lceil \tfrac{4n}{3}\right\rceil,\ldots,2n\big\}} P^o_{\omega} (X_{\ell} = b_+) \label{I-eq2}
\end{align}
where we used \eqref{I-eq1} in the third step and with the
short notation
\[
\widehat{\inf}_{\ell \in \big\{\left\lceil \tfrac{4n}{3} \right\rceil,\ldots,2n\big\}} P^x_{\omega}(X_{\ell}=y)
:=\inf_{\ell \in \big\{\left\lceil \tfrac{4n}{3}\right\rceil,\ldots,2n\big\}\cap \big(2\Z+ (x+y)\big)} P^x_{\omega}(X_{\ell}=y).
\]
Let us now have a closer look at the factors in the lower bound in \eqref{I-eq2} separately:\\[10pt]
\underline{First factor in \eqref{I-eq2}:}
We can bound the first factor from below by
\begin{align*}
& P^o_{\omega} \left(\tau(b_+) \le \tfrac{2n}{3},\ \tau(b_+) < \tau(b_-) \right) \vphantom{\exp \left( \max_{b_- \le i \le j \le b_+} \big(V(j) - V(i)\big)  \right)}\\
=\ & 1 - P^o_{\omega} \left(\tau(b_+) > \tfrac{2n}{3},\ \tau(b_+) < \tau(b_-) \right) - P^o_{\omega} \big( \tau(b_+) \ge \tau(b_-) \big)\vphantom{\exp \left( \max_{b_- \le i \le j \le b_+} \big(V(j) - V(i)\big)  \right)}\\
\ge\ & 1 - \tfrac{3}{2n}\cdot E^o_{\omega} \left[\tau(b_+) \cdot \mathbf{1}_{\{\tau(b_+) < \tau(b_-)\}} \right] - P^o_{\omega} \big( \tau(b_+) \ge \tau(b_-) \big)\vphantom{\exp \left( \max_{b_- \le i \le j \le b_+} \big(V(j) - V(i)\big)  \right)}\\
\ge\ & 1 - \tfrac{3}{2n}\cdot (b_+ - b_-)^2\cdot \exp \left( \max_{b_- \le i \le j \le b_+} \big(V(j) - V(i)\big)  \right) - \frac12 \ ,
\end{align*}
where we used \eqref{I-prel4} and assumption \eqref{I-ass4} for the last step.
Therefore, we get for \mbox{$\omega \in \Gamma(L,\delta)$} and $\exp\left(3 \delta L \right) \le n$ that
\begin{align}
&  P^o_{\omega} \left(\tau(b_+) \le \tfrac{2n}{3},\ \tau(b_+) < \tau(b_-) \right)  \ge \frac12 - \frac{3 \cdot 4 \cdot L^4}{2\cdot \exp(3\delta L)} \cdot \exp(2\delta L)= \frac12 - 6 \cdot L^4 \cdot \exp(-\delta L). \label{I-eq5}
\end{align}
\underline{Second factor in \eqref{I-eq2}:}
Due to Assumption \eqref{I-ass2} and to \eqref{I-eq1.1}, we get for $\omega \in \Gamma(L,\delta)$:
\begin{align}
&\frac{\mu_{\omega}(0)}{\mu_{\omega}(b_+)} =  \frac{\tfrac{1}{\omega_0}}{e^{-V(b_+)} + e^{-V(b_+-1)}} = \frac{\tfrac{1}{\omega_0}}{e^{-V(b_+)}\cdot (1+\rho_{b_+})} \nonumber \\
\ge\ & \frac{\tfrac{1}{1-\varepsilon}}{1+\tfrac{1-\varepsilon}{\varepsilon}} \cdot e^{V(b_+)}= \frac{\varepsilon}{1 - \varepsilon} \cdot e^{V(b_+)} \ge \frac{\varepsilon}{1 - \varepsilon} \cdot \exp(-\delta L). \label{I-eq3}
\end{align}
Here we used that $V(b_+) \ge -\delta L$ holds for $\omega \in \Gamma(L,\delta)$.\\[10pt]
\underline{Third factor in \eqref{I-eq2}:} 
For the last factor in \eqref{I-eq2}, we can compare the RWRE with the process $(\widetilde{X}_n)_{n \in \N_0}$ which behaves as the original RWRE but is reflected at the positions $T^-:=T^-(L)$ and $T^+:=~T^+(L)$, i.e.\ we have for $x \in \{T^-,T^-+1,\ldots,T^+\}$
\begin{align*}
& P_{\omega}^{x}(\widetilde{X}_0=x) = 1 \vphantom{\widehat{\inf}_{k \in \big\{\lceil \tfrac{\ell}{2}\rceil,\ldots,\ell\big\}}},\\
& P_{\omega}^{x}(\widetilde{X}_{n+1} = y\pm1|\widetilde{X}_n=y) = P_{\omega}^{x}(X_{n+1} = y\pm1|X_n=y)  \vphantom{\widehat{\inf}_{k \in \big\{\lceil \tfrac{\ell}{2}\rceil,\ldots,\ell\big\}}},\ \forall y \in \{T^-+1,\ldots,T^+-1\} \vphantom{\widehat{\inf}_{k \in \big\{\lceil \tfrac{\ell}{2}\rceil,\ldots,\ell\big\}}},\\
& P_{\omega}^{x}(\widetilde{X}_{n+1} = y+1|\widetilde{X}_n=y) = 1 \ \ \  \ \ \ \text{for } y= T^- \vphantom{\widehat{\inf}_{k \in \big\{\lceil \tfrac{\ell}{2}\rceil,\ldots,\ell\big\}}},\\
& P_{\omega}^{x}(\widetilde{X}_{n+1} = y-1|\widetilde{X}_n=y) = 1 \ \ \ \ \ \text{for } y= T^+. \vphantom{\widehat{\inf}_{k \in \big\{\lceil \tfrac{\ell}{2}\rceil,\ldots,\ell\big\}}}
\end{align*}
Therefore, we have for $\ell \in \big\{\left\lceil \tfrac{4n}{3}\right\rceil,\ldots,2n\big\}\cap \big(2\Z+ b_+\big)$
\begin{eqnarray}
P^{o}_{\omega} (X_{\ell} = b_+) \nonumber 
%\vphantom{\widehat{\inf}_{k \in \big\{\lceil \tfrac{\ell}{2}\rceil,\ldots,\ell\big\}}}\\
&\ge&  P^{o}_{\omega} (X_{\ell} = b_+,\ \min\{\tau(T^-),\tau(T^+)\} > 2n) \nonumber\\
% \vphantom{\widehat{\inf}_{k \in \big\{\lceil \tfrac{\ell}{2}\rceil,\ldots,\ell\big\}}} \\
%=\ & P^{o}_{\omega} (\widetilde{X}_{\ell} = b_+) - P^{o}_{\omega} (\widetilde{X}_{\ell} = b_+,\ \min\{\tau(T^-),\tau(T^+)\} \le 2n) \nonumber \vphantom{\widehat{\inf}_{k \in \big\{\lceil \tfrac{\ell}{2}\rceil,\ldots,\ell\big\}}}\\
&\ge&  P^{o}_{\omega} (\widetilde{X}_{\ell} = b_+) - P^{o}_{\omega} ( \min\{\tau(T^-),\tau(T^+)\} \le 2n) \nonumber\\
% \vphantom{\widehat{\inf}_{k \in \big\{\lceil \tfrac{\ell}{2}\rceil,\ldots,\ell\big\}}}\\
&\ge& P^{o}_{\omega} \big(\widetilde{X}_{\ell} = b_+, \ \tau(b_+) \le \tfrac{\ell}{2},\ \tau(b_+) < \tau(b_-)\big) - P^{o}_{\omega} ( \min\{\tau(T^-),\tau(T^+)\} \le 2n) \nonumber\\
% \vphantom{\widehat{\inf}_{k \in \big\{\lceil \tfrac{\ell}{2}\rceil,\ldots,\ell\big\}}}\\
&\ge& P^{o}_{\omega} \big(\tau(b_+) \le \tfrac{\ell}{2},\ \tau(b_+) < \tau(b_-)\big) \cdot \widehat{\inf\limits}_{k \in \big\{\left\lceil \tfrac{\ell}{2}\right\rceil,\ldots,\ell\big\}} P^{b_+}_{\omega} (\tilde X_{k} = b_+) \nonumber\\
% \vphantom{\widehat{\inf}_{k \in \big\{\lceil \tfrac{\ell}{2}\rceil,\ldots,\ell\big\}}}\\
&\ &\  - P^{o}_{\omega} ( \min\{\tau(T^-),\tau(T^+)\} \le 2n).  \label{I-eq4} 
%\vphantom{\widehat{\inf}_{k \in \big\{\lceil \tfrac{\ell}{2}\rceil,\ldots,\ell\big\}}}
\end{eqnarray}
Using \eqref{I-prel2} and \eqref{I-prel3}, we see that the last term in \eqref{I-eq4} with the negative sign decreases exponentially for $n \le e^{(1-2\delta)L}$, i.e.
\begin{align}
& P^{o}_{\omega} ( \min\{\tau(T^-),\tau(T^+)\} \le 2n) \le P^{o}_{\omega} \left( \min\{\tau(T^-),\tau(T^+)\} \le 2e^{(1-2\delta)L} \right) \nonumber \vphantom{\widehat{\inf}_{k \in \big\{\lceil \tfrac{\ell}{2}\rceil,\ldots,\ell\big\}}}\\
\le\ & P^{o}_{\omega} \left( \tau(T^-) \le 2\, e^{(1-2\delta)L} \right) + P^{o}_{\omega} \left( \tau(T^+) \le 2\, e^{(1-2\delta)L}\right) \nonumber \vphantom{\widehat{\inf}_{k \in \big\{\lceil \tfrac{\ell}{2}\rceil,\ldots,\ell\big\}}}\\
\le\ & 4  e^{(1-2\delta)L} e^{-L} = 4 e^{-2\delta L} \label{I-eq4a}. 
%\vphantom{\widehat{\inf}_{k \in \big\{\lceil \tfrac{\ell}{2}\rceil,%\ldots,\ell\big\}}}
\end{align}
In order to derive a lower bound for the first term in \eqref{I-eq4}, we first notice that the analogous calculation as in \eqref{I-eq5} shows for $\omega \in \Gamma(L,\delta)$ that
\begin{eqnarray}
P^o_{\omega} \left(\tau(b_+) \le \tfrac{\ell}{2},\ \tau(b_+) < \tau(b_-) \right) &\ge& 1 - \frac{2}{\ell} \cdot 4 \cdot L^4 \, e^{2\delta L} - \frac12 \nonumber\\
&\ge&  \frac12 - 6 \, L^4 \, e^{-\delta L} \label{I-eq5.1}
\end{eqnarray}
since $\ell \ge \left\lceil \tfrac{4n}{3}\right\rceil \ge \frac43 \,
e^{3 \delta L}$ for $n \ge e^{3 \delta L}$. For the second factor of \eqref{I-eq4}, we show the following\\

\begin{lem} \label{I-lem3}
For $\omega \in \Gamma(L,\delta)$ and for all $\ell \in 2 \N$, we have
\[
P^{b_+}_{\omega}(\widetilde{X}_{\ell} = b_+) \ge \frac12 \cdot \frac{1}{|T^-|+T^+ +1} \, e^{-\delta L}.
\]
\end{lem}
\begin{pfof}{Lemma \ref{I-lem3}}
Using the reversibility (cf.\ \eqref{I-eq1}) of $(\widetilde{X}_{\ell})_{\ell \in \N_0}$, we get
\begin{eqnarray}
 P^{b_+}_{\omega}(\widetilde{X}_{\ell} = b_+) &=&
 \sum_{x=T^-}^{T^+} P^{b_+}_{\omega}(\widetilde{X}_{\ell/2} = x) \cdot P^x_{\omega}(\widetilde{X}_{\ell/2} = b_+) \nonumber\\
&=& \sum_{x=T^-}^{T^+} P^{b_+}_{\omega}(\widetilde{X}_{\ell/2} = x) \cdot \frac{\widetilde{\mu}_{\omega}(b_+)}{\widetilde{\mu}_{\omega}(x)} \cdot P^{b_+}_{\omega}(\widetilde{X}_{\ell/2} = x), \label{I-eq7}
\end{eqnarray}
where $\widetilde{\mu}_{\omega}(\cdot)$ denotes a reversible
stationary measure of the reflected random walk $(\widetilde{X}_n)_{n \in \N_0}$ which is unique up to multiplication by a constant. To see that $(\widetilde{X}_{\ell})_{\ell \in \N_0}$ is also reversible, it is enough to note that $(\widetilde{X}_{\ell})_{\ell \in \N_0}$ can again be described as an electrical network with the following conductances:
\begin{align*}
\widetilde{C}_{(x,x+1)}(\omega)= \begin{cases}
C_{(x,x+1)}(\omega) = e^{-V(x)} & \text{for } x=T^-, T^-+1,\ldots, T^+-1 \\
0 & \text{for } x = T^- - 1, T^+.
\end{cases}
\end{align*}
Therefore, a reversible measure for the reflected random walk is given by (cf.\ \eqref{I-eq1.1})
\[
\widetilde{\mu}_{\omega}(x) = \begin{cases}\mu_{\omega}(x) = e^{-V(x)} + e^{-V(x-1)} & \text{for } x=T^-+1, T^-+2,\ldots, T^+-1, \\
e^{-V(T^-)} & \text{for } x=T^-, \\
e^{-V(T^+-1)} & \text{for } x=T^+.
\end{cases}
\]
Since $0 \le b_+ < T^+$, this implies
\begin{eqnarray}
\frac{\widetilde{\mu}_{\omega}(b_+)}{\widetilde{\mu}_{\omega}(x)}
 &\ge& \frac{e^{-V(b_+)}+ e^{-V(b_+-1)}}{e^{-V(x)}+ e^{-V(x-1)}} \nonumber \\
&\ge& \frac{e^{-V(b_+)}}{2\cdot e^{\left( - \min \{V(b_+),V(b_-) \} \right)}} \ge \frac{ e^{-\delta L}}{2} \label{I-eq8}
\end{eqnarray}
for $T^- \le x \le T^+$ and for $\omega \in \Gamma(L,\delta)$. By applying \eqref{I-eq8} to \eqref{I-eq7}, we get
\begin{eqnarray}
 P^{b_+}_{\omega}(\widetilde{X}_{\ell} = b_+)
&\ge&\frac12 \cdot \sum_{x=T^-}^{T^+} \left(P^{b_+}_{\omega}(\widetilde{X}_{\ell/2} = x)\right)^2 \cdot e^{-\delta L}\nonumber\\
&\ge& \frac12 \cdot \frac{1}{|T^-|+T^+ +1} \cdot e^{-\delta L}, \label{I-eq9}
\end{eqnarray}
by the Cauchy-Schwarz inequality since $\sum_{x=T^-}^{T^+} P^{b_+}_{\omega}(\widetilde{X}_{\ell/2} = x)=1$.
\renewcommand{\qedsymbol}{\hfill$\square$\vspace{1ex}}
\end{pfof}

We can now return to the proof of Proposition \ref{I-prop1} and finish our lower bound for the third factor in \eqref{I-eq2}. By applying \eqref{I-eq4a}, \eqref{I-eq5.1} and Lemma \ref{I-lem3} to \eqref{I-eq4}, we get for $e^{3\delta L} \le n \le e^{(1-2\delta) L}$ and $\omega \in \Gamma(L,\delta)$, since $|T^-|, T^+ \le L^2$,
\begin{align}
& \widehat{\inf}_{ \ell \in \big\{\left\lceil \tfrac{4n}{3}\right\rceil,\ldots,2n\big\}}P^{o}_{\omega} (X_{\ell} = b_+) \nonumber \vphantom{\frac12}\\
\ge\ &  \left(\frac12 - 6 \cdot L^4 \cdot e^{-\delta L}\right) \cdot \frac12 \cdot \frac{1}{2L^2 +1} e^{-\delta L} - 4 \cdot e^{-2\delta L} 
\ge\   e^{-\frac32\delta L}\label{I-eq12} \vphantom{\frac12}
\end{align}
for all $L=L(\delta)$ large enough.\\[10pt]
To finish the proof of Proposition \ref{I-prop1}, we can collect our lower bounds in \eqref{I-eq5}, \eqref{I-eq3}, and \eqref{I-eq12} and conclude with \eqref{I-eq2} that for $e^{3 \delta L} \le n \le e^{(1-2\delta)L}$ and for $\omega \in \Gamma(L,\delta)$ we have
\begin{eqnarray*}
P_{\omega} (X_{2n}=0)
&\ge& \left(\frac12 - 6 \cdot L^4 e^{-\delta L}\right) \cdot \frac{\varepsilon}{1-\varepsilon} e^{-\delta L} \cdot e^{-\frac32\delta L}\\
&\ge& e^{-3\delta L}
%  \vphantom{\frac12}
\end{eqnarray*}
for all $L=L(\delta)$ large enough. This shows \eqref{I-lem1} since we have $P_{\omega}(X_{2n}=0) \ge \varepsilon^{2n} > 0$ for all $n \in \N$ due to assumption \eqref{I-ass2}.
\end{pfof}
\begin{prop} \label{I-lemma2}
For $0 < \delta < 1$, we have
\begin{equation}
\label{I-lem2}
\p (\omega:\ \omega \in \Gamma(L,\delta) \text{ for infinitely many }L )=1.
\end{equation}
\end{prop}
\begin{pfof}{Proposition \ref{I-lemma2}}
Let $(B(t))_{t \in \R}$ be the two-sided Brownian motion from Theorem \ref{I-Komlos} and let us choose some $0 < \delta < \tfrac12$. For $y \in \R$ we define
$$\widehat{T}^+(y):= \inf\{t \ge 0:\ B(t)= y\}\ \ \mbox{and}\ \ 
\widehat{T}^-(y):= \sup\{t \le 0:\ B(t)= y\}$$
as the first hitting times of $y$ on the positive and negative side of the origin, respectively. Additionally, for $L \in \N$, $i \in \N$, $y \in \R$, we can introduce the following sets
$$
F_L^+(y):= \{\widehat{T}^+ \left(y \cdot L \right) < \widehat{T}^+ \left(- y \cdot L \right)\}\ \ \mbox{and}\ \ 
F_L^-(y):=  \{\widehat{T}^-\left(y \cdot L \right) < \widehat{T}^-\left(-y \cdot L \right)\}$$
on which the Brownian motion reaches the value $y \cdot L$ before $-y \cdot L$. Further we define
\begin{align*}
G_L^+(i):=\ & \left\{B(t) \ge (2i-1) \cdot \tfrac{\delta}{4} \cdot  L  \quad \text{for} \quad \widehat{T}^+\left(2i \cdot \tfrac{\delta}{4} \cdot L \right) \le t \le \widehat{T}^+\left((2i+2) \cdot \tfrac{\delta}{4} \cdot L \right) \right\}, \\
G_L^-(i):=\ & \left\{B(t) \ge (2i-1) \cdot \tfrac{\delta}{4} \cdot  L  \quad \text{for} \quad \widehat{T}^-\left((2i+2) \cdot \tfrac{\delta}{4} \cdot L \right) \le t \le \widehat{T}^-\left(2i \cdot \tfrac{\delta}{4} \cdot L \right) \right\}
\intertext{on which the Brownian motion does not decrease much between the first hitting time of the two levels of interest. Using these sets, we can define the sets}
A^+(L,\delta):=\ & F_L^+(\delta) \cap \left\{ \widehat{T}^+(1.1 \cdot L) \le L^2,\  \min_{\widehat{T}^+(\delta \cdot L) \le t \le \widehat{T}^+(1.1 \cdot L)} B(t) \ge \frac{\delta}{4} \cdot  L \right\}, \\
A^-(L,\delta):=\ & F_L^-(\delta) \cap \left\{ - \widehat{T}^-(1.1 \cdot L) \le L^2,\  \min_{\widehat{T}^-(1.1 \cdot L) \le t \le \widehat{T}^-(\delta \cdot L)} B(t) \ge \frac{\delta}{4} \cdot  L \right\}, \\
D^+(L,\delta):=\ & G_L^+(0) \cap G_L^+(1) \cap G_L^+(2) \allowdisplaybreaks[0]\\
& \cap \left\{ \widehat{T}^+(1.2 \cdot L) \le 0.9 \cdot L^2,\  \min_{\widehat{T}^+ \big(\tfrac{3 \cdot \delta}{2} \cdot L \big) \le t \le \widehat{T}^+(1.2 \cdot L)} B(t) \ge \frac{3\delta}{4} \cdot  L \right\}, \\
D^-(L,\delta):=\ & G_L^-(0) \cap G_L^-(1) \cap G_L^-(2) \\
& \cap \left\{ - \widehat{T}^-(1.2 \cdot L) \le 0.9 \cdot L^2,\  \min_{\widehat{T}^-(1.2 \cdot L) \le t \le \widehat{T}^-\big(\tfrac{3\delta}{2} \cdot L \big)} B(t) \ge \frac{3\delta}{4} \cdot  L \right\}
\end{align*}
which will be used for an approximation of our previously constructed valleys $\omega$ belonging to $\Gamma(L,\delta)$ which we illustrated in Figure \ref{I-figure1} on page \pageref{I-figure1}. Here, we added the factors $1.1$, $1.2$ and $0.9$ in contrast to the construction before in order to have some space for the approximation. For the Brownian motion, we can directly compute that we have
\begin{equation}
\label{I-eq15}
\p\big(D^+(1,\delta) \cap D^-(1,\delta)\big) > 0.
\end{equation}
Thereby, for all $L\in\N$, due to the scaling property of the Brownian motion, $( B(L^2 \cdot t)/L)_{ t \in \R}$
is again a two-sided Brownian motion with diffusion constant $\sigma$, this implies 
\begin{equation} \label{I-eq16}
\p\big(D^+(L,\delta) \cap D^-(L,\delta)\big) = \p\big(D^+(1,\delta) \cap D^-(1,\delta)\big) > 0.
\end{equation}
First, we notice that for $L_0 \in \N$ we have
\begin{align} \label{I-eq25}
& \p \left(\, \bigcap_{L=L_0}^{\infty} \Big( A^+(L,\delta) \cap A^-(L,\delta) \Big)^c \right)
\le \p \left(\ \bigcap_{k=\ell+1}^{\infty} \Big( A^+(L_k,\delta) \cap A^-(L_k,\delta) \Big)^c \right)
\end{align} 
for arbitrary $\ell \in \N_0$, where we define 
$$
L_k := \max \left\{10, \left\lceil \tfrac{2}{\delta} \right\rceil \right\} \cdot (L_{k-1})^2$$
for $k \in \N$ inductively. Note that for $n > \ell +1$ 
%where $\ell$ is chosen large enough such that
%\[
%(L_{k-1})^2 \le 0.1 \cdot L_{k} \quad \text{for all } k \ge \ell + 1
%\]
%and
with
\[
\mathcal{F}_n:=\sigma \left( \big(B(t)\big)_{-(L_{n-1})^2 \le t \le (L_{n-1})^2} \right),
\]
the following holds:
\begin{align}
& \p \left(\ \bigcap_{k=\ell+1}^{n} \Big( A^+(L_k,\delta) \cap A^-(L_k,\delta) \Big)^c \right) \nonumber \vphantom{\E \left[ \prod_{k=\ell+1}^{n-1} \mathbf{1}_{\big( A^+(L_k,\delta) \cap A^-(L_k,\delta)  \big)^c} \cdot \mathbf{1}_{\left\{ \max\limits_{-(L_{n-1})^2 \le t \le (L_{n-1})^2} |B(t)| \le (L_{n-1})^2 \right\}} \right.}\\
\le\ &  \E \left[ \prod_{k=\ell+1}^{n-1} \mathbf{1}_{\big( A^+(L_k,\delta) \cap A^-(L_k,\delta)  \big)^c} \cdot \mathbf{1}_{\left\{ \max\limits_{-(L_{n-1})^2 \le t \le (L_{n-1})^2} |B(t)| < (L_{n-1})^2 \right\}} \right. \nonumber\allowdisplaybreaks[0]\\
& \hspace{-5.5pt} \cdot \left. \left.  \E\left[ \vphantom{\prod_{k=\ell+1}^{n-1}} \mathbf{1}_{\left\{ \big(B(t + (L_{n-1})^2) - B((L_{n-1})^2) \big)_{t \in \R} \notin D^+ \left(L_n, \delta \right) \right\} \cup \left\{ \big(B(t - (L_{n-1})^2) - B(-(L_{n-1})^2) \big)_{t \in \R} \notin D^- \left(L_n, \delta \right) \right\}}  \right| \hspace{-1pt} \mathcal{F}_n \hspace{-1pt} \right] \hspace{-2pt} \right] \nonumber\allowdisplaybreaks[0]\\
& +  \p \left( \max\limits_{-(L_{n-1})^2 \le t \le (L_{n-1})^2} |B(t)| \ge (L_{n-1})^2 \right) \nonumber \vphantom{\E \left[ \prod_{k=\ell+1}^{n-1} \mathbf{1}_{\big( A^+(L_k,\delta) \cap A^-(L_k,\delta)  \big)^c} \cdot \mathbf{1}_{\left\{ \max\limits_{-(L_{n-1})^2 \le t \le (L_{n-1})^2} |B(t)| \le (L_{n-1})^2 \right\}} \right.}\\
\le\ &  \Big(1 - \p \Big(D^+ \left(L_n , \delta \right) \cap D^- \left(L_n, \delta \right) \Big) \Big) \cdot \p \left(\ \bigcap_{k=\ell+1}^{n-1} \Big( A^+(L_k,\delta) \cap A^-(L_k,\delta) \Big)^c \right) \nonumber \vphantom{\E \left[ \prod_{k=\ell+1}^{n-1} \mathbf{1}_{\big( A^+(L_k,\delta) \cap A^-(L_k,\delta)  \big)^c} \cdot \mathbf{1}_{\left\{ \max\limits_{-(L_{n-1})^2 \le t \le (L_{n-1})^2} |B(t)| \le (L_{n-1})^2 \right\}} \right.}\\
&  + \p \left( \max\limits_{-(L_{n-1})^2 \le t \le (L_{n-1})^2} |B(t)| \ge (L_{n-1})^2 \right) \nonumber \vphantom{\E \left[ \prod_{k=\ell+1}^{n-1} \mathbf{1}_{\big( A^+(L_k,\delta) \cap A^-(L_k,\delta)  \big)^c} \cdot \mathbf{1}_{\left\{ \max\limits_{-(L_{n-1})^2 \le t \le (L_{n-1})^2} |B(t)| \le (L_{n-1})^2 \right\}} \right.}\\
\le\ &  \Big(1 - \p \Big(D^+ \left(1 , \delta \right) \cap D^- \left(1, \delta \right) \Big) \Big)^{n-\ell}   + \sum_{k=\ell+1}^{n} \p \left( \max\limits_{-(L_{k-1})^2 \le t \le (L_{k-1})^2} |B(t)| \ge (L_{k-1})^2 \right) \label{I-eq26} \vphantom{\E \left[ \prod_{k=\ell+1}^{n-1} \mathbf{1}_{\big( A^+(L_k,\delta) \cap A^-(L_k,\delta)  \big)^c} \cdot \mathbf{1}_{\left\{ \max\limits_{-(L_{n-1})^2 \le t \le (L_{n-1})^2} |B(t)| \le (L_{n-1})^2 \right\}} \right.}.
\end{align}
To see that the first step in \eqref{I-eq26} holds, note that for 
\begin{align}
\omega \in &
\left\{\max\limits_{-(L_{n-1})^2 \le t \le (L_{n-1})^2)} |B(t)| < (L_{n-1})^2 \right\} \nonumber\\
& \cap \left\{\big(B(t + (L_{n-1})^2 ) - B((L_{n-1})^2) \big)_{t \in \R} \in D^+ \left(L_n, \delta \right) \right\} \label{I-eq26.x}
\end{align}
we have
\begin{eqnarray*}
 \min_{0 \le t \le (L_n)^2} B(t) &\ge& \min_{0 \le t \le (L_{n-1})^2} B(t) + \min_{(L_{n-1})^2 \le t \le (L_n)^2} B(t+(L_{n-1})^2) - B((L_{n-1})^2) \\
&>& - (L_{n-1})^2 - \frac{\delta}{4} \cdot L_n > - \delta \cdot L_n
\end{eqnarray*}
since $(L_{n-1})^2\le\delta L_n/2$ and
\begin{eqnarray*}
\max_{0 \le t \le (L_n)^2} B(t) &\ge& B\big((L_{n-1})^2\big) + \max_{(L_{n-1})^2 \le t \le (L_n)^2 - (L_{n-1})^2} B(t+(L_{n-1})^2) - B((L_{n-1})^2)\\
&\ge& - (L_{n-1})^2 + 1.2 \cdot L_n \ge 1.1 \cdot L_n
\end{eqnarray*}
since $(L_{n-1})^2\le L_n/10$.
In particular, we have $\widehat{T}^+(\delta \cdot L_n) < \widehat{T}^+(- \delta \cdot L_n)$ and $\widehat{T}^+(1.1 \cdot L_n) \le (L_n)^2$ on the considered set. Similarly, again on the set in \eqref{I-eq26.x}, we see that we have
\begin{eqnarray*}
\widehat{T}^+(\delta \cdot L_n) &>& \inf\{t \ge (L_{n-1})^2:\ (B(t + (L_{n-1})^2 ) - B((L_{n-1})^2) \ge \tfrac{\delta}{2} \cdot L_n\}, \\
  \widehat{T}^+(\delta \cdot L_n) &<& \inf\{t \ge (L_{n-1})^2:\ (B(t + (L_{n-1})^2 ) - B((L_{n-1})^2) \ge \tfrac{3\cdot \delta}{2} \cdot L_n\},
\end{eqnarray*}
since $(L_{n-1})^2\le\delta L_n/2$, this implies
$$ \min_{\widehat{T}^+(\delta \cdot L) \le t \le \widehat{T}^+(1.1 \cdot L)} B(t) \ge  \frac{\delta}{4} \cdot L_n$$
by construction of $D^+(L_n,\delta)$.
%\intertext{and similarly}
%& \max_{0 \le t \le \widehat{T}^+_b(L_{n})} B(t) \le \delta \cdot L_n,
%\intertext{and}
%& \widehat{T}^+(L_n) \stackrel{\text{def}}{=} \inf\{t \ge 0:\ B(t)= 1.1 \cdot L_n\}  \vphantom{\frac{\delta}{2}}\\
%\le\ &  (L_{n-1})^2 + \inf\{t \ge 0:\ B(t + (L_{n-1})^2) - B((L_{n-1})^2) = 1.2 \cdot L_n \} \vphantom{\frac{\delta}{2}}\\
%\le\ & (L_{n-1})^2 + 0.9 \cdot (L_n)^2 \le (L_n)^2 \vphantom{\frac{\delta}{2}},
%\end{align*}
Altogether, we can conclude that $\omega \in A^+(L_n,\delta)$ holds for our choice of $\omega$ in \eqref{I-eq26.x}. The argument for the negative part runs completely analogously. Further in \eqref{I-eq26}, we used the Markov property of the Brownian motion in the second step. Additionally, we iterated the first two steps $n-\ell-1$ times and used \eqref{I-eq16} for the last step. To control the last sum in \eqref{I-eq26}, let us recall
that due to the reflection principle (see e.g. Chapter III, Proposition 3.7
in \cite{RevuzYor}), we have
$$\forall T,x>0,\ \ \p\left(\max_{t\in[0,T]}\frac{B(t)}{\sigma\sqrt{T}}\ge x\right)=\p\left( |Z| \ge x \right)=2\p\left( Z \ge x \right) \le \frac{1}{x} \cdot \frac{e^{- \frac{x^2}{2}}}{\sqrt{2\pi}}
$$
for a random variable $Z \sim \mathcal{N}(0,1)$ (the last estimate can be found for example in Lemma 12.9 in Appendix B of \cite{morters}). Due to this upper bound, we can conclude that
\begin{align}
& \sum_{k=\ell+1}^{n} \p \left( \max\limits_{-(L_{k-1})^2 \le t \le (L_{k-1})^2} |B(t)| \ge (L_{k-1})^2 \right) \le 4 \cdot \sum_{k=\ell+1}^{n} \p \left( \max\limits_{0\le t \le (L_{k-1})^2} \frac{B(t)}{\sigma \cdot L_{k-1}}  \ge \frac{L_{k-1}}{\sigma} \right) \nonumber \\
\le\ & 8 \cdot \sum_{k=\ell+1}^{\infty} \frac{\sigma}{L_{k-1}} \cdot \frac{1}{\sqrt{2 \pi}} \cdot e^{- \frac{(L_{k-1})^2}{2 \sigma^2}} \xrightarrow{\ell \to \infty} 0. \label{I-eq27}
\end{align}
By combining the upper bounds in \eqref{I-eq25}, \eqref{I-eq26}, and \eqref{I-eq27}, we get for all $\ell \in \N_0$
\begin{align*}
& \p \left( \omega \notin  \big( A^+(L,\delta) \cap A^-(L,\delta) \big) \text{ for all } L \ge L_0\right) \vphantom{\sum_{k=\ell+1}^{\infty}}\\
\le\ & \lim_{n \to \infty}  \Big(1 - \p \Big(D^+ \left(1 , {\delta}\right) \cap D^- \left(1,{\delta} \right) \Big) \Big)^{n-\ell} \vphantom{\sum_{k=\ell+1}^{\infty}} \\
&  + \sum_{k=\ell+1}^{\infty} \p \left( \max\limits_{-(L_{k-1})^2 \le t \le (L_{k-1})^2} |B(t)| \ge (L_{k-1})^2 \right) \xrightarrow{\ell \to \infty} 0.
\end{align*}
Since $L_0 \in \N$ was chosen arbitrarily, we can conclude that for $0 < \delta < \tfrac12$ we have
\[
\p \left(\omega:\  \omega \in  \big(A^+(L, \delta) \cap A^-(L, \delta)\big) \text{ for infinitely many } L \right) = 1.
\]
Using the Koml{\'o}s-Major-Tusn{\'a}dy strong approximation Theorem (cf.\ Theorem \ref{I-Komlos}), we see that for $0 < \delta < \tfrac12$ we have
\begin{align*}
& \left\{\omega:\  \omega \in  \big(A^+(L, \delta) \cap A^-(L, \delta)\big) \text{ for infinitely many } L \right\} \allowdisplaybreaks[0]\\
\subseteq\ &  \left\{\omega:\ \omega \in \Gamma(L,2\delta) \text{ for infinitely many }L  \right\}, 
\end{align*}
which is enough to conclude that \eqref{I-lem2} holds for all $0 < \delta < 1$.
\end{pfof}

With the help of Proposition \ref{I-prop1} and Proposition \ref{I-lemma2}, we can now turn to the proofs of 
Theorems \ref{I-Rthm1} and \ref{I-Rthm2} and Corollary \ref{I-Rthm3}:

\begin{pfof}{Theorem \ref{I-Rthm1}} For a fixed $0 \le \alpha < 1$, we choose $0 < \delta < \tfrac16$ such that
$\alpha < (1-5\delta)/(1-2\delta)$.
For $\omega \in \Gamma(L, \delta)$, the inequality in \eqref{I-lem1} implies that
\begin{align*}
\sum_{n \in \N} P_{\omega}(X_{2n}=0) \cdot n^{-\alpha} &\ge \sum_{\lceil e^{3\delta L}\rceil \le n \le \lfloor e^{(1-2\delta)L}\rfloor} P_{\omega}(X_{2n}=0) \cdot n^{-\alpha} \\
&\ge \Big(e^{(1-2\delta)L} - e^{3\delta L} - 1 \Big) \cdot C \cdot e^{-3\delta L} \cdot \left(e^{(1-2\delta)L}\right)^{-\alpha}  \vphantom{\sum_{\lceil e^{3\delta L}\rceil \le n \le \lfloor
e^{(1-2\delta)L}\rfloor}}\\
&= C \cdot  \Big(e^{(1-5\delta)L} - 1 - e^{-3\delta L}\Big) \cdot e^{-\alpha(1-2\delta)L}  \xrightarrow{L\to\infty} \infty \nonumber .
\end{align*}
Since Proposition \ref{I-lemma2} shows that for $\p$-a.e.\ environment $\omega$ we find $L$ arbitrarily large such that $\omega \in \Gamma(L,\delta)$, we can conclude that \eqref{I-thm1} holds for $\p$-a.e.\ environment $\omega$.
\end{pfof}

\begin{pfof}{Theorem \ref{I-Rthm2}}
For fixed $\alpha > 0$, we choose $\delta$ such that
$
0 < \delta < \min \left\{ \frac{1}{2+3\alpha}, \frac{1}{5} \right\}$, which yields $1-2\delta -3\alpha \delta > 0$ and $1- 2 \delta > 3 \delta$.
For $\omega \in \Gamma(L,\delta)$, the inequality in \eqref{I-lem1} implies
\begin{align*}
& \hphantom{\ge}\ \sum_{n \in \N} \Big( P_{\omega}(X_{2n}=0)\Big)^{\alpha} \ge \sum_{\lceil e^{3\delta L}\rceil \le n \le \lfloor
e^{(1-2\delta)L}\rfloor} \Big(P_{\omega}(X_{2n}=0)\Big)^{\alpha} \\
&\ge \Big(e^{(1-2\delta)L} - e^{3\delta L} - 1 \Big) \cdot \big(C \cdot e^{-3\delta L}\big)^{\alpha} \vphantom{\sum_{\lceil e^{3\delta L}\rceil \le n \le \lfloor e^{(1-2\delta)L}\rfloor}}\\
&=  C^{\alpha} \cdot \Big(e^{(1-2\delta -3\alpha \delta)L} - 
e^{(3\delta - 3 \alpha \delta) L}  - e^{ -3\alpha \delta L} \Big) \vphantom{\sum_{\lceil e^{3\delta L}\rceil \le n \le \lfloor
e^{(1-2\delta)L}\rfloor}} \xrightarrow{L\to\infty} \infty. 
\end{align*}
Again since Proposition \ref{I-lemma2} shows that for $\p$-a.e.\ environment $\omega$ we find $L$ arbitrarily large such that $\omega \in \Gamma(L,\delta)$, we can conclude that \eqref{I-thm2} holds for $\p$-a.e.\ environment $\omega$.
\end{pfof}

\begin{pfof}{Corollary \ref{I-Rthm3}}
Due to the independence of the environments $\omega^{(1)}, \omega^{(2)}, \ldots, \omega^{(d)}$, we can extend the proof of Proposition \ref{I-lemma2} to get
\begin{equation}
\label{I-eq16.2}
\p^{\otimes d} \left(\text{For infinitely many } L \in \N, \text{ we have} \ \omega^{(i)} \in \Gamma(L,\delta)\ \text{for } i=1,2,\ldots d \right)=1
\end{equation}
for all $0 < \delta < 1$.
Indeed \eqref{I-eq26} becomes 

\begin{align}
& \p \left(\forall k=\ell+1,\cdots,n,\ 
\exists i,\ \omega_i\not\in A^+(L_k,\delta) \cap A^-(L_k,\delta) \right) \nonumber\\
\le\ &  \Big(1 - \p \Big(\forall i,\ B^{(i)}\in D^+ \left(L_n , \delta \right) \cap D^- \left(L_n, \delta \right) \Big) \Big) \cdot \p \left(\forall k=\ell+1,\cdots,n-1,\ 
\exists i,\ \omega_i\not\in A^+(L_k,\delta) \cap A^-(L_k,\delta)    \right) \nonumber \\
&  + \p \left( \max_{i=1,\cdots,d}\max\limits_{-(L_{n-1})^2 \le t \le (L_{n-1})^2} |B^{(i)}(t)| \ge (L_{n-1})^2 \right).\nonumber
\end{align}
Now, using Proposition \ref{I-prop1}, we have for $(\omega^{(1)}, \omega^{(2)}, \ldots, \omega^{(d)})$ with $\omega^{(i)} \in \Gamma(L,\delta)$ for $i=1,2,\ldots d$ 
\begin{align*}
& \sum_{n \in \N} \prod_{k=1}^{d} P_{\omega^{(k)}}(X_{2n}=0) 
 \ge  \sum_{\lceil e^{3\delta L}\rceil \le n \le \lfloor e^{(1-2\delta)L}\rfloor} \prod_{k=1}^{d} P_{\omega^{(k)}}(X_{2n}=0) \\
&\ge  \Big(e^{(1-2\delta)L} - e^{3\delta L} - 1 \Big) \cdot C^d \cdot e^{-3 \delta d L}\\
&= C^d \cdot  \Big(e^{(1-2\delta-3 \delta d)L}  - e^{(3\delta -3\delta d) L} - e^{-3 \delta d L} \Big) \xrightarrow{L\to\infty} \infty
\end{align*}
for $0 < \delta < \frac{1}{2+3d}$.
Since \eqref{I-eq16.2} holds for arbitrarily small $\delta$, we can conclude that \eqref{I-eqcor2} holds for $\p^{\otimes d}$-a.e.\ environment $(\omega^{(1)}, \omega^{(2)}, \ldots, \omega^{(d)})$.
\end{pfof}
\section{Recurrence properties of the RWRE (I)-(III)} \label{I-sec1.6}

\subsection{Direct products involving a one dimensional RWRE}

\begin{prop}[Case (I)]\label{I-cor6}
Fix a random environment $\omega$ which fulfils \eqref{I-ass1} and \eqref{I-ass2}. Let $(X_n,Y_n)_{n \in \N_0}$ be a 2-dimensional process where $(X_n)_{n\in \N_0}$ and $(Y_n)_{n\in \N_0}$ are independent with respect to $P_{\omega}$, $(X_n)_{n\in \N_0}$ being a RWRE in the environment $\omega$ (in the sense of \eqref{I-RWRE}) and $(Y_n)_{n\in \N_0}$ a centered random walk such that $Y_0=0$ and $(Y_n/A_n)_n$
converges in distribution to a $\beta$-stable distribution 
with $\beta\in(1,2]$ (for some suitable normalization $A_n$).

Then, $(X_n,Y_n)_{n \in \N_0}$ is recurrent for $\p$-a.e.\ environment $\omega$.
\end{prop}

\begin{pfof}{Proposition \ref{I-cor6}}
Due to the local limit theorem, we have
$P(Y_{d_0n}=0)\sim C(A_n)^{-1}$ for some $C>0$
and for $d_0:=gcd\{m\ge 1,\, \p(X_m=Y_m=0)\ne 0\}$.
Recall that $A_n=n^{\frac 1\beta}L(n)$ with $L$
a slowly varying function.
Due to the independence of the two components, 
we have
$$
\sum_{n \in \N} P_{\omega}\big((X_{n},Y_{n})=(0,0)\big) = \sum_{n \in \N} P_{\omega}\big(X_{d_0n}=0\big) \cdot P_{\omega}\big(Y_{d_0n}=0\big) 
 = \infty,$$
where the last equation is due to Theorem \ref{I-Rthm1} applied with $\frac 1\beta<\alpha<1$.
%due to Theorem \ref{I-Rthm1} applied with $\frac 1\beta<\alpha<1$ for the last two steps. 
This proves the recurrence of the process $(X_n,Y_n)_{n \in \N_0}$ for $\p$-a.e.\ environment $\omega$.
\renewcommand{\qedsymbol}{$\square$}
\end{pfof}

Observe that, if we take $(X_n)_n$ to be the simple symmetric random walk on $\mathbb Z$ in Proposition \ref{I-cor6} instead of Sinai's 
walk, we have $\mathbb P(X_{2n}=0)\sim cn^{-\frac 12}$ and hence
we lose the recurrence as soon as $\beta<2$.
\subsection{Other two-dimensional RWRE governed by a one-dimensional RWRE}
We study now the cases (II) and (III). 
We consider a process moving horizontally with
probability $\delta$ and vertically with probability
$1-\delta$. We assume that the horizontal displacements
follow Sinai's  walk and that the vertical ones either follow some
recurrent random walk (case (II)) or depend on the parity
of the first coordinate of the current position (case (III)).

\begin{prop}[Case (II)] \label{I-cor1.5.7}
Let $\delta\in(0,1)$.
Let  $\omega=(\omega_x)_x$ be a random environment which fulfils \eqref{I-ass1} and \eqref{I-ass2}. We assume that, given $\omega$, $(M_n)_{n \in \N_0}$ is a Markov chain with values in $\Z^2$ such that
\begin{align*}
&P_{\omega}\big(M_0=(0,0)\big) = 1 \vphantom{\frac{1 - \delta}{2}},\\
&P_{\omega}\big(M_{n+1}=(x+1,y)\big|M_{n}=(x,y)\big)=\delta \cdot \omega_x \vphantom{\frac{1 - \delta}{2}},\\
&P_{\omega}\big(M_{n+1}=(x-1,y)\big|M_n=(x,y)\big)=\delta \cdot (1-\omega_x) \vphantom{\frac{1 - \delta}{2}},\\
&P_{\omega}\big(M_{n+1}=(x,y+z)\big|M_{n}=(x,y)\big)=(1 - \delta)\cdot\nu(z),
\end{align*}
where $\nu$ is a probability distribution on $\mathbb Z$
such that $(\nu^{*n}(A_n\cdot))_n$ converges to
a $\beta$-stable distribution with $\beta\in(1,2]$ (for some suitable increasing 
sequence $(A_n)_n$ of positive real numbers).
Then, $(M_{n})_{n \in \N_0}$ is recurrent for $\p$-a.e.\ environment $\omega$.
\end{prop}
Let us recall that $\nu^{*n}(A_n\cdot)$ is the distribution of $(Z_1+\dots+Z_n)/A_n$ 
if $Z_1,...,Z_n$ are iid random variables with distribution $\nu$.

\begin{pfof}{Proposition \ref{I-cor1.5.7}}
Let us write $M_n=(\tilde X_n,\tilde Y_n)$.
We look at the process $(M_{n})_{n \in \N_0}$ whenever it has moved in the first component. For this, we define inductively
$\tau_0:=0$ and $\tau_k:= \inf\left\{n > \tau_{k-1}:\, \tilde X_n \neq \tilde X_{\tau_{k-1}} \right\}$ for $k \ge 1$.
Additionally, we define
${X}_n:= \tilde X_{\tau_n}$ and ${Y}_n:=\tilde Y_{\tau_n}$
for $n \in \N_0$. 
Note that $({X}_n)_{n \in \N_0}$ is a usual RWRE on $\Z$ 
with environment $\omega$. Further, we have
$${Y}_n=\sum_{k=1}^{\tau_n-n}Z_k
=\sum_{\ell=1}^n\sum_{k=\tau_{\ell-1}-\ell+2}^{\tau_\ell-\ell}Z_k,$$
where $(Z_k)_k$ is a sequence of i.i.d. random variables with
distribution $\nu$.
We know that the random variables $(\tilde Z_\ell:=\sum_{k=\tau_{\ell-1}-\ell+2}^{\tau_\ell-\ell}Z_k)_\ell$ are identically distributed and centered. Let us prove
that their distribution belongs to the domain of attraction 
of a $\beta$-stable distribution. 
We know that $(\sum_{k=1}^mZ_k/A_m)_m$
converges in distribution to a $\beta$-stable centered random variable $U$ and that 
$A_m=m^{\frac 1\beta}L(m)$, $L$ being
a slowly varying function.
Observe that $(Y_n/A_{\tau_n-n})_n$ converges in distribution to $U$ and that
$(A_{\tau_n-n}/A_n)_n$ converges almost surely to ${\mathbb E[\tau_1-1]}^{\frac 1\beta}$.
Hence $(Y_n/A_n)_n$ converges
in distribution to ${\mathbb E[\tau_1-1]}^{\frac 1\beta}U$.
Therefore, (since $\p(\tilde Z_1=0)>0$) we conclude that $P_\omega(Y_{n}=0)\sim C(A_n)^{-1}$.
Hence, for $\p$-a.e.\ environment $\omega$, we have
$$
\sum_{n \in \N} P_{\omega}\big(({X}_{2n},{Y}_{2n})=(0,0)\big) = \sum_{n \in \N}
 P_{\omega}({X}_{2n}=0) \cdot P_{\omega}({Y}_{2n}=0) =\infty$$
(due to Theorem \ref{I-Rthm1} applied with $\frac 1\beta<\alpha<1$).
This implies the recurrence of $({X}_{n},{Y}_{n})_{n}$ and so of $(M_n)_{n}$.
\renewcommand{\qedsymbol}{$\square$}
\end{pfof}

Finally we consider the case (III).
We suppose now that every vertical line is oriented upward if
the line is labelled by an even number and downward otherwise.
We consider again a process moving horizontally with
probability $\delta$ and moving vertically with probability
$1-\delta$. We assume that the horizontal displacements
follow a Sinai walk (as in the previous example)
but that the vertical displacements follow the orientation of the vertical line on which the walker
is located. 
For a probability measure $\nu$, we write $\tilde\nu:=\nu(-\cdot)*\nu$ for the distribution of
$Z_2-Z_1$ if $Z_1$ and $Z_2$ are independent with distribution $\nu$.
\begin{prop}[Case (III), odd-even orientations of vertical lines] \label{I-cor***}
Let $\delta\in(0,1)$ and
let $\omega$ be a random environment which fulfils \eqref{I-ass1} and \eqref{I-ass2}. Given $\omega$, $(M_n)_{n \in \N_0}$ is a Markov chain with values in $\Z^2$ such that
\begin{align*}
&P_{\omega}\big(M_0=(0,0)\big) = 1 \vphantom{\frac{1 - \delta}{2}},\\
&P_{\omega}\big(M_{n+1}=(x+1,y)\big|M_{n}=(x,y)\big)=\delta \cdot \omega_x \vphantom{\frac{1 - \delta}{2}},\\
&P_{\omega}\big(M_{n+1}=(x-1,y)\big|M_{n}=(x,y)\big)=\delta \cdot (1-\omega_x) \vphantom{\frac{1 - \delta}{2}},\\
&P_{\omega}\big(M_{n+1}=(x,y+z)\big|M_{n}=(x,y)\big)=(1 - \delta)\cdot\nu((-1)^xz),
\end{align*}
with $\nu$ a probability distribution on $\mathbb Z$  such that
$(\tilde\nu^{*n}(A_n\cdot))_n$ converges to
a $\beta$-stable distribution with $\beta>1$ (for a suitable increasing sequence $(A_n)_n$ of positive 
real numbers).
Then, $(M_n)_{n \in \N_0}$ is recurrent for $\p$-a.e.\ environment $\omega$.
\end{prop}
\begin{pfof}{Proposition \ref{I-cor***}}
The proof follows the same scheme as the previous one and uses the same notations 
$(\tilde{X}_n,\tilde Y_n)_n$, $\tau_n$ and $(X_n,Y_n)$.
Again $({X}_n)_{n \in \N_0}$ is a RWRE on $\Z$ with environment $\omega$.
Let us write $T_n:=\tau_n-\tau_{n-1}$, $\tau_0^+=\tau_0^-=0$, $\tau_n^{+}:=
\sum_{\ell=1}^nT_{2\ell-1}$ and $\tau_n^{-}:=
\sum_{\ell=1}^nT_{2\ell}$. Observe that $\tau_n^+$ (resp. $\tau_n^-$) is the number
of vertical moves on an even (resp. odd) vertical axis before the $2n$-th horizontal
displacement.
We have
$${Y}_{2n}=\sum_{\ell=1}^{n}[\xi_{2\ell-1}-\xi_{2\ell}],
\ \ \mbox{with}\ \ \xi_{2\ell-1}=\sum_{k=\tau^+_{\ell-1}-\ell+2}^{\tau^+_\ell-\ell}Z_{2k-1},\ \ \xi_{2\ell}=\sum_{k=\tau^-_{\ell-1}-\ell+2}^{\tau^+_\ell-\ell}Z_{2k},$$
where $(Z_k)_k$ is a sequence of i.i.d. random variables with
distribution $\nu$.
With these notations $Z_{2k+1}$ (resp. $-Z_{2k}$)  is the $k$-th vertical
displacement on an even (resp. odd) vertical axis.

The random variables $\xi_{2\ell-1}-\xi_{2\ell}$ are
iid. We already know that $\xi_{1}-\xi_{2}$ is centered. Let us prove that
its distribution belongs to the domain of attraction of a $\beta$-stable
centered distribution, i.e. that $ Y_{2n}$
suitably normalized converges to a $\beta$-stable random variable.
We observe that
$$ Y_{2n}=\sum_{k=1}^{\tau_n^+-n}Z_{2k-1}-\sum_{k=1}^{\tau_n^--n}Z_{2k}=U_n+V_n^+-V_n^-+W_n,$$
with
$$U_n:=\sum_{k=1}^{n\mathbb E[\tau_1-1]}(Z_{2k-1}-Z_{2k}), $$
$$V_n^+:=\sum_{k=1}^{\tau_n^+-n}(Z_{2k-1}-\mathbb E[Z_1])- \sum_{k=1}^{n\mathbb E[\tau_1-1]}(Z_{2k-1}-\mathbb E[Z_1])  ,$$
$$ V_n^-:=\sum_{k=1}^{\tau_n^--n}(Z_{2k}-\mathbb E[Z_1])- \sum_{k=1}^{n\mathbb E[\tau_1-1]}(Z_{2k}-\mathbb E[Z_1])  ,$$
$$W_n:=(\tau_n^+-\tau_n^-)\mathbb E[Z_1].$$
We know that $(U_n/A_n)_n$ converges in distribution
to a $\beta$-stable random variable $U$.
We observe that
$(V_n^\pm/A_n)_n$ 
converges in probability to 0 (since $A_{|\tau_n^+-n\mathbb E[\tau_1]|}\ll A_n$). 

If $\mathbb E[Z_1]=0$, we conclude that $(Y_n/A_n)_n$ converges in distribution to $U$. 

Assume now that $\mathbb E[Z_1]\ne 0$. Then 
$ \left(W_n/\sqrt{n}\right)_n$
is independent of $(U_n)_n$ and converges
in distribution to some centered normal variable $W$ (assumed to be independent of $U$).

Hence, if $1<\beta<2$, we conclude that $(Y_{2n}/A_n)_n$
converges in distribution to $U$.

If $\beta=2$ and $\mathbb E[Z_1]\ne 0$, 
we can choose $A_n$ such that $U$ and $W$ have the same distribution and
we conclude that  $(Y_{2n}/\sqrt{n+A_n^2})_n$
converges in distribution to a $U$.

Hence, for $\p$-a.e.\ environment $\omega$, we have
$$
\sum_{n \in \N} P_{\omega}\big(({X}_{2n},{Y}_{2n})=(0,0)\big) = \sum_{n \in \N} P_{\omega}({X}_{2n}=0) 
\cdot P_{\omega}({Y}_{2n}=0) =\infty,$$
due to Theorem \ref{I-Rthm1} applied with $\alpha>1/\beta$
and due to the local limit theorem for
$({Y}_{2n})_n$. This implies the recurrence of $({X}_{n},{Y}_{n})_{n}$ and so of $(M_{n})_{n}$, for
$\p$-a.e.\ environment $\omega$.
\renewcommand{\qedsymbol}{$\square$}
\end{pfof}
\begin{prop}
If we replace Sinai's walk by the simple symmetric random walk on 
$\mathbb Z$ (i.e. if we replace $\omega_x$ by $1/2$) in the
assumptions of Propositions  \ref{I-cor1.5.7} and \ref{I-cor***},
then the walk $(M_n)_n$ is recurrent if and only if $\sum_n\frac{1}{A_n\sqrt{n}}=\infty$. 

In particular it is transient as soon as $\beta<2$.
\end{prop}
\begin{proof}
We follow the proofs of Propositions  \ref{I-cor1.5.7} and \ref{I-cor***} and we use the fact
that $P_\omega({X}_{2n}=0)$ is equivalent to $c/\sqrt{n}$ for some $c>0$
as $n$ goes to infinity.
We have
$$\sum_n P_\omega(M_n=0)=\sum_nP_\omega( X_{2n}=0)\mathbb P 
      \left(\exists K\in\{0,...,\tau-1\},\ 
        Y_{2n}+\sum_{k=1}^K Z_k=0\right),$$
where $\tau$ has the same distribution as $\tau_1$ and where
$(Z_k)_k$ is a sequence of iid random variables with distribution $\nu$
such that $ Y_{2n}$, $\tau$ and $(Z_k)_k$ are independent.
Now observe that 
$$\mathbb P 
      \left(\exists K\in\{0,...,\tau-1\},\ 
       Y_{2n}+\sum_{k=1}^K Z_k=0\right)\ge \mathbb P 
      ( Y_{2n}=0)\sim \frac{C_1}{A_n}$$
for some $C_1>0$ due to the local limit theorem for $(Y_{2n})_n$ and that
$$\mathbb P 
      \left(\exists K\in\{0,...,\tau-1\},\ 
        Y_{2n}+\sum_{k=1}^K Z_k=0\right)\le
\mathbb P\left(| Y_{2n}|\le\sum_{k=1}^{\tau-1}|Z_k|\right)$$
\begin{eqnarray*}
&\le& \sum_{m\ge 0}\mathbb P(| Y_{2n}|=m)\mathbb P
\left(\sum_{k=1}^{\tau-1}|Z_k|\ge m\right)\\
&\le & \frac {C_2}{A_n}\mathbb E\left[\sum_{k=1}^{\tau-1}|Z_k|\right]
\le  \frac {C_2}{A_n}\mathbb E[\tau]\mathbb E[|Z_1|],
\end{eqnarray*}
for $C_2>0$ using the uniform bound given by the local limit theorem.
Hence 
$$ \mathbb P 
      \left(\exists K\in\{0,...,\tau-1\},\ 
        Y_{2n}+\sum_{k=1}^K Z_k=0\right)\approx \frac 1 {A_n}.$$
\end{proof}
{\bf Acknowledgements} NG and MK thank Daniel Boivin for the invitations to  Brest, the ANR project MEMEMO2 (ANR-10-BLAN-0125) for supporting these visits and the department of mathematics, Universit\'e de Brest, for its hospitality. FP acknowledges support of the ANR project MEMEMO2 (ANR-10-BLAN-0125).

\bibliographystyle{alpha}

\end{document}